\documentclass[12pt]{amsart}
\pdfoutput=1
\usepackage{enumerate}
\usepackage{enumitem}
\usepackage[hyperindex, breaklinks, colorlinks=true, linkcolor=blue, urlcolor=blue, citecolor=blue, anchorcolor=blue, pdfborder={0 0 0}]{hyperref}
\usepackage{url}
\usepackage{graphicx,color}
\usepackage{xcolor}
\usepackage{cite}
\usepackage{amsthm, amsmath, amssymb, amsaddr}
\usepackage{mathtools}
\usepackage[top=45truemm, bottom=45truemm, left=30truemm, right=30truemm]{geometry}
\usepackage{nicefrac}
\usepackage{cancel}
\usepackage{float}
\usepackage{tabularx}
\usepackage{makecell}
\usepackage{multirow}
\usepackage{array}
\usepackage{ragged2e}
\usepackage{hyperref}
\usepackage{authblk}
\usepackage{booktabs}
\usepackage{xcolor}

\newcolumntype{P}[1]{>{\RaggedRight\hspace{0pt}}p{#1}}
\newcolumntype{L}{>{\begin{math}}l<{\end{math}}}%
\newcolumntype{C}{>{\begin{math}}c<{\end{math}}}%
\newcolumntype{R}{>{\begin{math}}r<{\end{math}}}%

\newtheorem{theorem}{Theorem}[section]

\newtheorem{definition}{Definition}[section]

\newtheorem{example}{Example}[section]

\theoremstyle{definition}

\DeclareMathOperator{\gde}{gde}

\title[Rotating Binaries]{\Large Rotating Binaries}

\begin{document}

\clearpage\maketitle
\thispagestyle{empty}

\begin{center}
Anant Gupta$^{1}$, Idriss J. Aberkane$^{2}$, Sourangshu Ghosh$^{3}$, Adrian Abold$^{4}$,\\Alexander Rahn$^{5}$ and Eldar Sultanow$^{6,7}$\\
\vspace{10px}
\hspace*{5mm}
\fontsize{8pt}{10pt}\selectfont
$^{1}$~Georgia Institute of Technology, North Ave NW, Atlanta, GA 30332, USA, \href{mailto:agupta886@gatech.edu}{agupta886@gatech.edu}\\
\vspace{5px}
$^{2}$~Unesco-Unitwin Complex Systems Digital Campus, Chair of Prof. Pierre Collet, ICUBE - UMR CNRS 7357, 4 rue Kirschleger, 67000 Strasbourg, France, \href{mailto:idriss.aberkane@polytechnique.edu}{idriss.aberkane@polytechnique.edu}\\
\vspace{5px}
$^{3}$~Department of Civil Engineering, Indian institute of Technology Kharagpur, Kharagpur, West Bengal 721302, India, \href{mailto:sourangshu@iitkgp.ac.in}{sourangshu@iitkgp.ac.in}\\
\vspace{5px}
$^{4}$~Friedrich-Alexander-Universität Erlangen-Nürnberg, Lange Gasse 20, 90403 Nuremberg, Germany, \href{mailto:adrian.abold@fau.de}{adrian.abold@fau.de}\\
\vspace{5px}
$^{5}$~Nuremberg Institute of Technology, Keßlerpl. 12, 90489 Nuremberg, Germany, \href{mailto:rahnal71212@th-nuernberg.de}{rahnal71212@th-nuernberg.de}\\
\vspace{5px}
$^{6}$~Potsdam University, Chair of Business Informatics, Processes and Systems, Karl-Marx Straße 67, 14482, Potsdam, Germany, \href{mailto:eldar.sultanow@wi.uni-potsdam.de}{eldar.sultanow@wi.uni-potsdam.de}\\
\vspace{5px}
$^{7}$~Capgemini, Bahnhofstraße 30, 90402, Nuremberg, Germany, \href{mailto:eldar.sultanow@capgemini.com}{eldar.sultanow@capgemini.com}\\
    
\end{center}
\begin{abstract}
This paper investigates the behaviour of rotating binaries. A rotation by $r$ digits to the left of a binary number $B$ exhibits in particular cases the divisibility \mbox{$l\mid N_1(B)\cdot r+1$}, where $l$ is the bit-length of $B$ and $N_1(B)$ is the Hamming weight of $B$, that is the number of ones in $B$. The integer $r$ is called the \textit{left-rotational distance}. We investigate the connection between this rotational distance, the length and the Hamming weight of binary numbers. Moreover we follow the question under which circumstances the above mentioned divisibility is true. We have found out and will demonstrate that this divisibility occurs for $kn+c$ cycles.
\end{abstract}

\section{Introduction}
\label{sec:introduction}
A divisibility feature of rotated binary numbers has been discovered by Darrell Cox \cite{Ref_Cox_2021} and taken further, analyzed and visualized for numerous cases using the Python programming language by Eldar Sultanow \cite{Ref_Sultanow_2021}. In the following we will develop a computational base for the binary rotation, its related cycles and generalize the divisibility feature.

Let us take a binary number $B$ of length $l$ with $N_1(B)$ ones (and $N_0(B)=l-N_1(B)$ zeros), for example $l=8$, $N_1(B)=5$ and $B=10110101=181$, the minimum that is obtainable by rotating $B$ is $B_{\min}=01011011=91$ and the maximum is $B_{\max}=11011010=218$. The left-rotational distance is $r=3$, since we obtain the maximum $11011010$ by three left rotates of the minimum $01011011$. The maximum $218=(91\cdot2^3)\bmod(255)$ can be obtained directly using equation~\ref{eq:b_max} follows:

\begin{equation}
\label{eq:b_max}
\begin{array}{l}
B_{\max}=(B_{\min}\cdot2^r)\bmod{(2^l-1)} \\
B_{\min}=(B_{\max}\cdot2^{l-r})\bmod{(2^l-1)} 
\end{array}
\end{equation}

Vice versa, we calculate the minimum directly as $91=(218\cdot2^{8-3})\bmod(255)$. Moreover, we can calculate the length $l$ (See Sedgewick and Wayne \cite[p.~185]{Ref_Sedgewick_2011}) and the Hamming weight $N_1(B)$ using $B_{\max}$ (see Weisstein \cite{Ref_Weisstein_DigitCount} and Allouche and Shallit \cite[p.~74]{Ref_Allouche_2003}) directly:

\[
\begin{array}{l}
l=\lfloor\log_2(B_{\max})\rfloor+1\\
N_1(B)=B_{\max}-\gde(B_{\max}!,2)=B_{\max}-\sum_{i=1}^{l-1}\left\lfloor\nicefrac{B_{\max}}{2^i})\right\rfloor
\end{array}
\]

It is briefly mentioned that $\gde(n,2)$ denotes the greatest dividing exponent of base $2$ with respect to a number $n$, which is the largest integer value of $k$ such that $2^k\mid n$ with $2^k\le n$, see \cite{Ref_Weisstein_GDE}.

By applying these formulas to our example, we obtain $l=\lfloor\log_2(218)\rfloor+1=7+1=8$ and $N_1(B)=218-\gde(218!,2)=218-(109+54+27+13+6+3+1)=218-213=5$. The divisibility $l\mid N_1(B)\cdot r+1$ can be written in our example as $8\mid \left(218-\sum_{i=1}^{7}\left\lfloor\nicefrac{218}{2^i}\right\rfloor\right) \cdot 3+1$. In our example the divisibility $l\mid N_1(B)\cdot r+1$ holds, since $8\mid 5\cdot 3+1$ is true. Our question is: Under which circumstances is this divisibility generally granted?

\par\medskip
To prove that this divisibility holds, we need to show that there always exist integers $a$ and $b$ that solves the diophantine equation, which we deduce from equation~\ref{eq:b_max}:
\[
\left(\frac{(2^l-1)b+B_{\max}}{B_{\min}}\right)^{N_1(B)}=2^{a\cdot l-1}=2^{r\cdot N_1(B)}
\]

\par\medskip
In our example $a=b=2$ provide a solution: to solve $218=(91\cdot2^r)\bmod{(2^8-1)}$ we substitute $2^r=Y$ and solve the linear congruence $(91\cdot Y)\equiv218\bmod{(2^8-1)}$, which is solvable if $\gcd(91,2^8-1)\mid218$ and there is a unique solution if $91$ and the modulus $2^8-1$ are coprime $\gcd(91,2^8-1)=1$. This coprimality is given here. The solution is $Y\equiv8\bmod(255)$ and resubstitution of $Y$ leads to $2^r\equiv8\bmod(255)$, which brings us to the solutions $r=3,11,19,27,\ldots$ and so on. All these $r$ values enables us to find solutions $a=2,7,12$ for $5\cdot r=8\cdot a-1$. In order that the divisibility $l\mid N_1(B)\cdot r+1$ is given, we must show that $\gcd(B_{\min},2^L-1)\mid B_{\max}$ and $2^l-1\mid2^rB_{\min}-B_{\max}$.

\section{\texorpdfstring{Binary rotations lead us to $3n+c$ cycles}{Binary rotations lead us to 3n+c cycles}}
\label{sec:binary_rotations_to_cycles}
Take a binary number $B$ with a Hamming weight $N_1(B)$ as input for a function $z$, which Darrel Cox \cite{Ref_Cox_2021} defined as follows, where $0\le x_1<x_2<\ldots<x_{N_1(B)}\le N_1(B)-1$ are the positions (indexing is zero-based) in $B$ occupied by $1$:
\begin{equation}
\label{eq:z}
z(B)=\sum_{i=1}^{N_1(B)}3^{N_1(B)-i}2^{x_i}
\end{equation}

This function $z$ is adapted from Halbeisen's and Hungerbühler's function $\varphi$, see \cite{Ref_Halbeisen_Hungerbuehler_1997}. In the introductory example $B_{\max}=11011010$ we have $z(B_{\max})=z(11011010)=319$ and the five positions in our binary number $B_{\max}$ that are occupied by 1 are $(x_1,x_2,x_3,x_4,x_5)=(0,1,3,4,6)$:

\begin{flalign*}
319&=3^{N_1(B)-1}2^{x_1}+3^{N_1(B)-2}2^{x_2}+3^{N_1(B)-3}2^{x_3}+3^{N_1(B)-4}2^{x_4}+3^{N_1(B)-5}2^{x_5}\\
&=3^42^0+3^32^1+3^22^3+3^12^4+3^02^6    
\end{flalign*}

\par\medskip
Similarly we can calculate $z(B_{\min})=z(01011011)=842$. Both integers, the $319$ and the $864$ belong to a $3n+13$ cycle that is given by the following function whose parameter in this case is $c=2^l-3^{N_1(B)}=13$:

\begin{equation}
\label{eq:func_collatz}
f_c(x)=
\begin{cases}
\nicefrac{3x+c}{2}	&	2\nmid x\\
\nicefrac{x}{2}		&	\text{otherwise}
\end{cases}
\end{equation}

\par\medskip
Note that $319$ is the smallest member and $864$ is the largest member of this sequence and the binary representation of $B_{\max}=11011010$ reflects the course of this cycle starting with its smallest member $319$, where the ones represent odd members and the zeros represent even members:

\[
(v_1,v_2,v_3,v_4,v_5,v_6,v_7,v_8)=(319,485,734,367,557,842,421,638)
\]
\par\medskip
Table~\ref{table:rot_distances_5_2} shows the left-rotational distances of a binary number that we obtain from the integer $x$ in the first column using the reverse function $z^{-1}(x)$ to another number located in the same row of $v$. For example the left-rotational distance of $557=z(10101101)$ to $734=z(01101011)$ is six, which we highlighted blue. Table~\ref{table:rot_distances_5_2} highlights our case of the rotational distance from $842=z(01011011)=z(B_{\min})$ to $319=z(11011010)=z(B_{\max})$ using the color green. The integer $r=3$ is the only rotational distance value that provides a solution for the divisibility $8\mid 5\cdot r+1$.

\begin{table}[H]
	\centering
	\begin{tabular}{L|LLLLLLLL}
		\thead{} &
		\thead{\boldsymbol{319}} &
		\thead{\boldsymbol{485}} &
		\thead{\boldsymbol{734}} &
		\thead{\boldsymbol{367}} &
		\thead{\boldsymbol{557}} &
		\thead{\boldsymbol{842}} &
		\thead{\boldsymbol{421}} &
		\thead{\boldsymbol{638}}\\
		\hline
		\thead{\boldsymbol{319}} &
		0 & 1 & 2 & 3 & 4 & 5 & 6 & 7
		\\
		\thead{\boldsymbol{485}} &	
		7 & 0 & 1 & 2 & 3 & 4 & 5 & 6
		\\
		\thead{\boldsymbol{734}} &
		6 & 7 & 0 & 1 & 2 & 3 & 4 & 5
		\\
		\thead{\boldsymbol{367}} &
		5 & 6 & 7 & 0 & 1 & 2 & 3 & 4
		\\
		\thead{\boldsymbol{557}} &
		4 & 5 & \colorbox{blue!15}{\textbf{6}} & 7 & 0 & 1 & 2 & 3
		\\
		\thead{\boldsymbol{842}} &
		\colorbox{green!30}{\textbf{3}} & 4 & 5 & 6 & 7 & 0 & 1 & 2
		\\
		\thead{\boldsymbol{421}} &
		2 & 3 & 4 & 5 & 6 & 7 & 0 & 1
		\\
		\thead{\boldsymbol{638}} &
		1 & 2 & 3 & 4 & 5 & 6 & 7 & 0
		\\
	\end{tabular}
	\caption{Rotational distances of the $3x+13$ cycle members}
	\label{table:rot_distances_5_2}
\end{table}

\section{What we know about cycles}
Starting point of our considerations is the function $f_c(x)$ given by equation~\ref{eq:func_collatz}.

Let $S$ be a set containing two elements $n_1$ and $n_0$, which are bijective functions over $\mathbb{Q}$:
\begin{equation}
n_1(x)=\nicefrac{3x+c}{2}\hspace{4em} n_0(x)=\nicefrac{x}{2}
\end{equation}

Let a binary operation be the left-to-right composition of functions $n_1\circ n_0$, where $n_1\circ n_0(x)=n_0(n_1(x))$. $S^\ast$ is the composition monoid (transformation monoid), which is freely generated by $S$. The identity element is the identity function $id_{\mathbb{Q}}=e$. We call $e$ an \textit{empty string}. $S^\ast$ consists of all expressions (strings) that can be concatenated from the generators $n_1$ and $n_0$. Every string can be written in precisely one way as product of factors $n_1$ and $n_0$ and natural exponents $k_i>0$:

\[
e,n_1^{k_1},n_0^{k_1},n_1^{k_1}n_0^{k_2},n_0^{k_1}n_1^{k_2},n_1^{k_1}n_0^{k_2}n_1^{k_3},n_0^{k_1}n_1^{k_2}n_0^{k_3},\ldots
\]

These uniquely written products are called \textit{reduced words} over $S$. Using exponents $k_i,h_i>0$, we construct strings $s_i=n_1^{k_i}n_0^{h_i}$ and concatenate these to a larger string:

\[
s_1s_2\cdots s_l=n_1^{k_1}n_0^{h_1}n_1^{k_2}n_0^{h_2}\cdots n_1^{k_l}n_0^{h_l}
\]

Note that each string $s_i$ is a reduced word, since $k_i,h_i>0$. Let us evaluate this (large) string by inputting a natural number $v_1$. If the result is again $v_1$ then we obtain a cycle:

\[
n_1^{k_1}n_0^{h_1}n_1^{k_2}n_0^{h_2}\cdots n_1^{k_l}n_0^{h_l}(v_1)=n_0^{h_l}(n_1^{k_l}(\cdots n_0^{h_2}(n_1^{k_2}(n_0^{h_1}(n_1^{k_1}(v_1))))))=v_1
\]

We write the sums briefly as $N_1=k_1+\cdots+k_l$ and $N_0=h_1+\cdots+h_l$. The cycle contains $N_1+N_0$ elements. We summarize this fact to the following definition~\ref{def:odd_even_elements}:

\begin{definition}
\label{def:odd_even_elements}
A cycle consists of $N_1+N_0$ elements, where $N_1=k_1+\cdots+k_l$ is the number of its odd members and $N_0=h_1+\cdots+h_l$ the number of its even members.
\end{definition}

\par\noindent
Moreover we define $A=a_1+\cdots+a_l$ with
\[
a_i=2^{\sum_{j=1}^{i-1}k_j+h_j}\cdot\left(3^{k_i}-2^{k_i}\right)\cdot 3^{\sum_{j=i+1}^{l}k_j}
\]

\par\medskip\noindent
Theorem~\ref{theo:v1} calculates the smallest member of the $3n+c$ cycle, which in line with definition~\ref{def:odd_even_elements} consists of $N_1$ odd and $N_0$ even members \cite{Ref_Gupta_2020}:

\begin{theorem}
\label{theo:v1}
The smallest number $v_1$ belonging to a cycle having $N_1$ odd and $N_0$ even members is:
\[
v_1=\frac{c\cdot A}{2^{N_1+N_0}-3^{N_1}}
\]
\end{theorem}

\par\medskip
\begin{example}
\label{ex:C_3_11}
We consider a $3n+11$ cycle that has $N_1+N_0=8+6=14$ elements and choose $(k_1,k_2,k_3,k_4)=(3,1,3,1)$ and $(h_1,h_2,h_3,h_4)=(1,1,2,2)$. Its smallest element is $v_1=13$ and we obtain all elements by evaluating the strings: $v_2=n_1(v_1)$, $v_3=n_1(v_2)$, $v_4=n_1(v_3)$ and $v_5=n_0(v_4)$ and so forth. It applies:
\begin{flalign*}
&n_1n_1n_1n_0\circ n_1n_0\circ n_1n_1n_1n_0n_0\circ n_1n_0n_0(v_1)\\
=&n_1^3n_0\circ n_1n_0\circ n_1^3n_0^2\circ n_1n_0^2(v_1)\\
=&s_1\circ s_2\circ s_3\circ s_4(v_1)=v_1
\end{flalign*}

\par\noindent
This cycle is $(v_1,v_2,v_3,\ldots,v_{14})=(13,25,43,70,35,58,29,49,79,124,62,31,52,26)$. We calculate $v_1$ directly as follows:

\[
v_1=\frac{11\cdot 11609}{2^{8+6}-3^8}=\frac{11\cdot11609}{9823}=13
\]

\par\noindent
In this case $11609=A=a_1+a_2+a_3+a_4=4617+1296+3648+2048$:

\[
\begin{array}{llll}
a_1=2^{0}&(3^3-2^3)&3^{1+3+1}&=4617\\
a_2=2^{3+1}&(3^1-2^1)&3^{3+1}&=1296\\
a_3=2^{3+1+1+1}&(3^3-2^3)&3^{1}&=3648\\
a_4=2^{3+1+1+1+3+2}&(3^1-2^1)&3^{0}&=2048
\end{array}
\]
\end{example}

\begin{theorem}
\label{theo:max}
The maximum odd element in a $3n+c$ cycle occurs immediately before the maximum even element.
\end{theorem}

\begin{proof}
The maximum even element of the cycle cannot succeed an even element as the preceding element would be twice the element taken. The maximum odd element occurs before the maximum even element is equivalent to saying that the maximum even element follows the maximum odd element. Let $v_1,v_2$ be odd elements in the cycle with $v_1>v_2$, then the elements after $v_1,v_2$ will be $w_1=\nicefrac{3v_1+c}{2}$ and $w_2=\nicefrac{3v_2+c}{2}$. Since $v_1>v_2$ and $w_1>w_2$, the element after the maximum odd element is greater than the element after any other odd element. Therefore the maximum odd element of the cycle precedes the maximum even element of the cycle.
\end{proof}

In conformity with definition~\ref{def:odd_even_elements}, let us consider a $3n+c$ cycle $(v_1,v_2,\ldots,v_l)$ consisting of $N_1$ odd integers and $N_0$ even integers. Let us consider a binary parity vector (it is synonymous to a binary sequence or binary non-reduced word) consisting of $l=N_1+N_0$ elements, which has a $1$ at position $i$, if $v_i$ is odd, and otherwise $0$. Theorem~\ref{theo:cycle_restriction_1} specifies several cycle restrictions:

\begin{theorem}
\label{theo:cycle_restriction_1}
For a $3n+c$ cycle with $N_1$ odd and $N_0$ even members applies:

\begin{enumerate}[label=(\alph*)]
\item A cycle only exists if the inequality $2^{N_1+N_0}-3^{N_1}>0$ holds.
\item The condition for the existence of a cycle can be detailed as follows \cite{Ref_Cox_2012}: A cycle only exists if $c\mid2^{N_1+N_0}-3^{N_1}$.
\item Let $0\le x_1<x_2<\ldots<x_{N_1}\le N_1-1$ be all positions (the indexing is zero-based) in the parity vector occupied by $1$. A $3n+c$ cycle only exists if the divisibility $2^{N_1+N_0}-3^{N_1}\mid c\cdot z(s)$ holds, where $z$ is the function~\ref{eq:z}.
\item The number of $3n+c$ cycles is alsways less than or equal to the number of $3n+a\cdot c$ cycles, where $a$ is an odd number.
\end{enumerate}
\end{theorem}

\begin{example}
We refer to the $3n+11$ cycle $(13,25,43,70,35,58,29,49,79,124,62,31,52,26)$ again. The corresponding parity vector is $(1,1,1,0,1,0,1,1,1,0,0,1,0,0)$ and the non-reduced word is $n_1n_1n_1n_0n_1n_0n_1n_1n_1n_0n_0n_1n_0n_0$.

The indices are $(x_1,\ldots,x_8)=(0,1,2,4,6,7,8,11)$ and therefore $z(8)=3^72^0+3^62^1+3^52^2+3^42^4+3^32^6+3^22^7+3^12^8+3^02^{11}=11609$.

\par\medskip\noindent
Correctly it applies that $2^{8+6}-3^8\mid11\cdot11609$, more specifically it is $9.823\mid127.699$ and $9.823\cdot13=127.699$. 
\end{example}

\begin{theorem}
\label{theo:cycle_uniqueness}
Two different primitive cycles, $3n+c_1$ and $3n+c_2$, can never share a common parity vector.
\end{theorem}

\begin{proof}
A $3n+c$ cycle with a given parity vector first appears at:
\[
c=\frac{2^{N_1+N_0}-3^{N_1}}{\gcd(A,2^{N_1+N_0}-3^{N_1})}
\]

Let there exist cycles $3n+c_1$ and $3n+c_2$ with the same parity vector, this implies that the values of $A$ and $2^{N_1+N_0}-3^{N_1}$ as defined in Definition \ref{def:odd_even_elements} are same for both the cycles. Therefore using the formula, a cycle can exist iff $v_1$ is an integer, id est $c \cdot A$ divides $2^{N_1+N_0}-3^{N_1}$. The cycle will originate for the minimum such value of $c$. Therefore there can only be one value of $c$ for which the parity vector produces a cycle that is not inherited.
\end{proof}

\section{Boundary features of cycles}
\label{sec:boundary_features}
Halbeisen and Hungerbühler \cite{Ref_Halbeisen_Hungerbuehler_1997} introduced a boundary feature for cycles as function $M(l,n)$, where $l$ is the cycle length and $n$ the number of its odd members. Let $S_{l,n}$ denote the set of all binary words of length $l$ containing exactly $n$ ones and otherwise only zeros. This set contains exactly $\binom{l}{n}$ words -- exactly the number of ways in which we may select $n$ elements out of $l$ total where the order is irrelevant. In Halbeisen's and Hungerbühler's notation, the Hamming weight is denoted by $n$, which corresponds to our notation $N_1$ following Wolfram Math \cite{Ref_Weisstein_DigitCount}, that is $n=N_1$. In the example given by table~\ref{table:calculation_M_5_2}, the elements of the set $S_{5,2}$ are all listed in the first column.

The second column of table~\ref{table:calculation_M_5_2} contains all binary words that result from left-rotating the binary word $B$ in the first column up to $l$ times:
\[
\lambda_{left}(B,5),\lambda_{left}(B,1),\lambda_{left}(B,2),\lambda_{left}(B,3),\lambda_{left}(B,4)
\]
In generalized terms, this set is denoted as $\sigma(B)=\{\lambda_{left}(B,i):1\le i\le l\}$. Remembering that $z$ is the function~\ref{eq:z}, the third column of table~\ref{table:calculation_M_5_2} contains the corresponding output of this function when inputting the rotated binary words:
\[
z(\lambda_{left}(B,5)),z(\lambda_{left}(B,1)),z(\lambda_{left}(B,2)),z(\lambda_{left}(B,3)),z(\lambda_{left}(B,4))
\]
The last column contains the minimum of these values. Finally, the largest of all these minima is $M(5,2)$ or generally, see \cite{Ref_Halbeisen_Hungerbuehler_1997}:

\begin{equation}
\label{eq:M_l_n}
M(l,n)=\max_{B\in S_{l,n}}\{\min_{t\in\sigma(B)}z(t)~\}
\end{equation}

Additionally to Halbeisen's and Hungerbühler's boundary feature $M(l,n)$ Darrell Cox et al. \cite{Ref_Cox_2021} introduced another boundary feature as a function $N(l,n)$. Let $g=\gcd(l,n)$, the function $N(l,n)$ is defined as follows:
\begin{equation}
\label{eq:N_l_n}
N(l,n)=2\cdot M(l,n)-\sum_{i=0}^{g-1}2^{i\cdot n/g}3^{n-1-i\cdot n/g}
\end{equation}

\begin{example}
\label{ex:M_N}
We choose a cycle given by $f_c(x)$ of length $l=5$ having $n=2$ odd members, where $c=2^l-3^n=2^5-3^2=23$. Let us choose the binary words $11000$ and $10100$ and calculate the smallest member of the corresponding cycle in each case.

In the first case, namely $11000$ synonymous with $n_1n_1n_0n_0n_0=n_1^2n_0^3=n_1^{k_1}n_0^{h_1}$ we obtain $v_1=\nicefrac{c\cdot A}{2^{N_1+N_0}-3^{N_1}}=\nicefrac{23\cdot 5}{2^{2+3}-3^2}=5$. The resulting cycle is $(5,19,40,20,10)$ which is given by the first row and third column in table~\ref{table:calculation_M_5_2}.

In the second case, $10100$ that is synonymous with $n_1n_0n_1n_0n_0=n_1^1n_0^1n_1^1n_0^2=n_1^{k_1}n_0^{h_1}n_1^{k_2}n_0^{h_2}$ we obtain $v_1=\nicefrac{23\cdot 7}{2^{2+3}-3^2}=7$. The resulting cycle is $(7,22,11,28,14)$ which is given by the second row and third column in table~\ref{table:calculation_M_5_2}.

Table~\ref{table:calculation_M_5_2} exhibits how $M(l,n)$ is calculated, which in our concrete case is $M(5,2)=7$. Additionally we calculate $N(5,2)=2\cdot M(5,2)-2^03^{2-1-0}=14-3=11$.
\end{example}

\begin{table}[H]
	\centering
	\begin{tabular}{L|LLLL}
		\thead{} &
		\thead{\textbf{word }\boldsymbol{s}} &
		\thead{\textbf{set }\boldsymbol{\sigma(s)}\textbf{ of left rotated words}} &
		\thead{\boldsymbol{\{z(t):t\in\sigma(s)\}}} &
		\thead{\boldsymbol{\displaystyle \min_{t\in\sigma(s)}z(t)}}\\
		\hline
		\thead{\boldsymbol{1}} &
		11000 &
		11000,10001,00011,00110,01100 &
		5,19,40,20,10 &
		5
		\\
		\thead{\boldsymbol{2}} &		
		10100 &
		10100,01001,10010,00101,01010 &
		7,22,11,28,14 &
		7
		\\
		\thead{\boldsymbol{3}} &
		10010 &
		10010,00101,01010,10100,01001 &
		11,28,14,7,22 &
		7
		\\
		\thead{\boldsymbol{4}} &
		10001 &
		10001,00011,00110,01100,11000 &
		19,40,20,10,5 &
		5
		\\
		\thead{\boldsymbol{5}} &
		01100 &
		01100,11000,10001,00011,00110 &
		10,5,19,40,20 &
		5
		\\
		\thead{\boldsymbol{6}} &
		01010 &
		01010,10100,01001,10010,00101 &
		14,7,22,11,28 &
		7
		\\
		\thead{\boldsymbol{7}} &
		01001 &
		01001,10010,00101,01010,10100 &
		22,11,28,14,7 &
		7
		\\
		\thead{\boldsymbol{8}} &
		00110 &
		00110,01100,11000,10001,00011 &
		20,10,5,19,40 &
		5
		\\
		\thead{\boldsymbol{9}} &
		00101 &
		00101, 01010, 10100, 01001, 10010 &
		28,14,7,22,11 &
		7
		\\
		\thead{\boldsymbol{10}} &
		00011 &
		00011,00110,01100,11000,10001 &
		40,20,10,5,19 &
		5
		\\
		\hline
		\multicolumn{4}{r}{The largest of all minimum $z$ values is $M(l,n)=M(5,2)=$} &
		7
		\\
	\end{tabular}
	\caption{Calculation of $M(5,2)$}
	\label{table:calculation_M_5_2}
\end{table}

\section{Constructing one cycle from another}
Cycles may interrelate, which means they have the same length and an equal amount of odd members. We refer to example~\ref{ex:M_N} and consider the $3n+23$ cycle $(5,19,40,20,10)$. A cycle, which interrelates to this $3n+23$ cycle is for example the $3n+69$ cycle $(15,57,120,60,30)$.

If we go back to example~\ref{ex:C_3_11}, then we can provide two interrelated cycles as well. For $l=N_1+N_0=8+6=14$ we obtain $c=2^l-3^{N_1}=2^{14}-3^8=9823$ and the $3n+9823$ cycle is $(11609,22325,38399,62510,31255,51794,25897,43757,70547,110732,55366,\\27683,46436,23218)$.

When we divide the parameter $c$ and all cycle members by $893$, then we obtain the reduced interrelated $3n+11$ cycle $(13,25,43,70,35,58,29,49,79,124,62,31,52,26)$.

\begin{theorem}
\label{theo:containment_M_N}
Let a $3n+c$ cycle of length $l=N_1+N_0$ has $N_1$ odd and $N_0$ even members, where $c=2^l-3^{N_1}$. It always applies that $M(l,n)$ is greater than the smallest member and $N(l,n)$ is less than the largest odd member of this cycle. Recall that we can take $N_1$ and $n$ to be synonymous, since Hungerbühler and Halbeisen denote the Hamming weight by $n$.

If $c$ is divisible by an odd integer $a$, then for the (reduced) interrelated $3n+\nicefrac{c}{a}$ cycle it applies that $\nicefrac{M(l,n)}{a}$ is greater than the smallest member and $\nicefrac{N(l,n)}{a}$ is less than the largest odd member of this reduced cycle.
\end{theorem}

\section{\texorpdfstring{Constant sums of $3n+c$ cycle members}{Constant sums of 3n+c cycle members}}
Let us consider the set of all possible binary words of the length $l=5$. This set contains $2^5=32$ elements. There exist $8$ different periodic sequences, whereby we consider two sequences to be the \textit{same}, if one of them can be obtained by left or right rotations from the other. Therefore different sequences do not share any sequence member. The members of these different sequences does not depend from the binary word's value, but from its length $l$ and Hamming weight.

Now, let us regard a set of \textit{same} (periodic) sequences and the number of its members is not equal to $l$. In this case the members of these sequences additionally depend from the left-rotational distance $r$ of $B_{\min}$ to $B_{\max}$. In those cases the set may contain $\nicefrac{l}{2}$, $\nicefrac{l}{3}$ or $2l$ sequences. 

\par\medskip
Let $B$ be a binary number of length $l$ and with a Hamming weight $N_1(B)$. We use this binary number $B$ to create a $3n+c$ sequence $(v_1,v_2,\ldots,v_l)$ by performing left-rotations and applying the function $z$ as we did in section~\ref{sec:boundary_features}:

\[
v_1=z(\lambda_{left}(B,1)),v_2=z(\lambda_{left}(B,2)),\ldots,v_l=z(\lambda_{left}(B,l)) 
\]

\par\medskip
Moreover we define a function $Z$ that uses the binary number $B$ as an input and yields the sum of all the members belonging to the periodic sequence $(v_1,v_2,\ldots,v_l)$ which we generated from $B$:

\begin{equation}
\label{eq:N_l_n}
Z(B)=\sum_{i=1}^{l}v_{i}=\sum_{i=1}^{l}z(\lambda_{left}(B,i))
\end{equation}

\begin{example}
\label{ex:trivial_case} 
We choose $B=00001$ and this results in $B_{\min}=00001$ and $B_{\max}=10000$. The length $l=5$ and the Hamming weight $N_1(B)=1$. This is a really trivial case of periodic sequence generation. Each row in Table~\ref{table:trivial_case} depicts the periodic sequence which we generated from $B$. This Table~\ref{table:trivial_case} illustrates that the generated cycles are not only reflected (horizontally) by rows, but also (vertically) by columns. That is because the digit $1$ appears on every position within our rotated binary number $B$ only once. Finally always $Z(B)=2^l-1=31$. Note that $B_{\min}$ and $B_{\max}$ are identical for the (rotated) $B$ in each table row, since rotating a binary number generally does not affect the corresponding $B_{\min}$ and $B_{\max}$. 
Halbeisen's and Hungerbühler's set $S_{l,n}$ which we introduced in section~\ref{sec:boundary_features} contains in the present case $l$ words: $S_{l,n}=S_{5,1}=\binom{5}{1}=5$. This behavior is exactly the same for $N_1(B)=l-1$, since the binary number $11110$ behaves in the same way as $00001$.

\begin{table}[H]
	\centering
	\begin{tabular}{LLL|LLLLLL}
		\multicolumn{3}{c|}{\thead{B}} &
		\thead{\textbf{$v_{1}$}} &
		\thead{\textbf{$v_{2}$}} &
		\thead{\textbf{$v_{3}$}} &
		\thead{\textbf{$v_{4}$}} &
		\thead{\textbf{$v_{5}$}} &
		\thead{$Z(B)$}\\
		\hline
		00001 &
		= &
		1 &
		16 &
		8 &
		4 &
		2 &
		1 &
		31 
		\\		
		00010 &
		= &
		2 &
		8 &
		4 &
		2 &
		1 &
		16 &
		31 
		\\
		00100 &
		= &
		4 &
		4 &
		2 &
		1 &
		16 &
		8 &
		31 
		\\
		01000 &
		= &
		8 &
		2 &
		1 &
		16 &
		8 &
		4 &
		31 
		\\
		10000 &
		= &
		16 &
		1 &
		16 &
		8 &
		4 &
		2 &
		31 
		\\
		\hline
		\multicolumn{3}{r}{Z(B) =} &
		31 & 31 & 31 & 31 & 31\\
	\end{tabular}
	\caption{Trivial Case for $l=5$ and $N_1(B)=1$}
	\label{table:trivial_case}
\end{table}
\end{example}

Now let us consider cases that are not such trivial. The number of possible Hamming weights is odd for a binary word having an even length. For instance, if the binary number's length $l=4$ then this binary number can have a Hamming weight $N_1(B)\in\{0,1,2,3,4\}$.
If the binary number's length $l=6$ then this binary number can have a Hamming weight $N_1(B)\in\{0,1,2,3,4,5,6\}$.

\begin{example}
\label{ex:nontrivial_case_u_2}
We choose $B=001001$ which results in $B_{\min}=001001$ and $B_{\max}=100100$. The length $l=6$ is even and the Hamming weight is $N_1(B)=2$. Table \ref{table:nontrivial_case_u_2} shows that for $N_1(B)=2$ the periodic sequence which we generated from $B$ represents a concatenation of the cycle $(44,22,11)$, in which this cycle occurs exactly twice. In other words, this sequence has $N_1(B)=2$ periods. We have $\nicefrac{l}{N_1(B)}=\nicefrac{6}{2}=3$ distinct words and therefore $3$ distinct members in this periodic sequence which we generated from $B$.

When we invert the binary number $B=001001$ by replacing $0$ with $1$ (and vice versa) we obtain the binary number $110110$. This inverted binary number has the Hamming weight $N_1(B)=l-2=4$ and exhibits the same behavior as $B=001001$. Generally spoken, the cases for $N_1(B)=l-2$ behave as same as $N_1(B)=2$.

\begin{table}[H]
	\centering
	\begin{tabular}{LLL|LLLLLLLLL}
		\multicolumn{3}{c|}{\thead{B}} &
		\thead{\textbf{$v_{1}$}} &
		\thead{\textbf{$v_{2}$}} &
		\thead{\textbf{$v_{3}$}} &
		\thead{\textbf{$v_{4}$}} &
		\thead{\textbf{$v_{5}$}} &
		\thead{\textbf{$v_{6}$}} &
		\thead{$Z(B)$}\\
		\hline
		001001 &
		= &
		9 &
		44 &
		22 &
		11 &
		44 &
		22 &
		11 &
		154
		\\	
		010010 &
		= &
		18 &
		22 &
		11 &
		44 &
		22 &
		11 &
		44 &
		154
		\\	
		100100 &
		= &
		36 &
		11 &
		44 &
		22 &
		11 &
		44 &
		22 &
		154
		\\
		\hline
		\multicolumn{3}{r}{$\nicefrac{1}{2}\cdot Z(B)=$} &
		 77 & 77 & 77 & 77 & 77 & 77
	\end{tabular}
	\caption{Non-trivial case for $l=6$ and $N_1(B)=2$}
	\label{table:nontrivial_case_u_2}
\end{table}
\end{example}

\begin{example}
\label{ex:nontrivial_case_u_3}
We choose $B=010101$ which results in $B_{\min}=010101$ and $B_{\max}=101010$. The length is again $l=6$ and the Hamming weight is $N_1(B)=3$. Table \ref{table:nontrivial_case_u_3} shows that for $N_1(B)=3$ the periodic sequence which we generated from $B$ represents a concatenation of the cycle $(74,37)$, in which this cycle occurs exactly three times. In other words, this sequence has $N_1(B)=3$ periods. Here we have $\nicefrac{l}{N_1(B)}=\nicefrac{6}{3}=2$ distinct words and therefore $2$ distinct members in this periodic sequence which we generated from $B$.

Also here, inverting the binary number $B$ leads to the same behavior, id est the cases for $N_1(B)=l-3$ behave as same as $N_1(B)=3$.

\begin{table}[H]
	\centering
	\begin{tabular}{LLL|LLLLLLLLL}
		\multicolumn{3}{c|}{\thead{B}} &
		\thead{\textbf{$v_{1}$}} &
		\thead{\textbf{$v_{2}$}} &
		\thead{\textbf{$v_{3}$}} &
		\thead{\textbf{$v_{4}$}} &
		\thead{\textbf{$v_{5}$}} &
		\thead{\textbf{$v_{6}$}} &
		\thead{$Z(B)$}\\
		\hline
		010101 &
		= &
		21 &
		74 &
		37 &
		74 &
		37 &
		74 &
		37 &
		333
		\\	
		101010 &
		= &
		42 &
		37 &
		74 &
		37 &
		74 &
		37 &
		74 &
		333
		\\
		\hline
		\multicolumn{3}{r}{$\nicefrac{1}{3}\cdot Z(B)=$} &
		 111 & 111 & 111 & 111 & 111 & 111
	\end{tabular}
	\caption{Non-trivial case for $l=6$ and $N_1(B)=3$}
	\label{table:nontrivial_case_u_3}
\end{table}
\end{example}

It is important to note in conclusion that the amount of cycles which we can generate from a given binary number $B$ is deterministic and not random. The Hamming weight $N_1(B)$ affects the binary combinatorics and it affects together with the length $l$ the amount of possible cycles that we are able to generate from $B$.

For a given binary number $B$ of length $l$ with a Hamming weight $N_1(B)$ the cases behave as same as for the inverted binary number (having the Hamming weight $l-N_1(B)$). For a given binary number $B$ the possibilities for generating periodic sequences from $B$ is limited as well.

\section{\texorpdfstring{Generalizations to $kn+c$ cycles}{Generalization to kn+c cycles}}
First we generalize the function~\ref{eq:func_collatz} by introducing the following function:
\begin{equation}
\label{eq:func_collatz_k}
f_{k,c}(x)=
\begin{cases}
\nicefrac{kx+c}{2}	&	2\nmid x\\
\nicefrac{x}{2}		&	\text{otherwise}
\end{cases}
\end{equation}

\par\medskip\noindent
\textbf{\textit{Generalization of theorem~\ref{theo:v1}}}
\par\noindent
We can generalize theorem~\ref{theo:v1} by replacing $3$ by $k$. The smallest number $v_1$ belonging to a cycle $kn+c$ cycle having $N_1$ odd and $N_0$ even members is:
\[
v_1=\frac{c\cdot A}{(k-2)(2^{N_1+N_0}-k^{N_1})}
\]

\par\medskip\noindent
\textbf{\textit{Generalization of theorem~\ref{theo:max}}}
\par\noindent
Theorem~\ref{theo:max} applies equally to $kn+c$ cycles as it does to $3n+c$ cycles. The proof provided for theorem~\ref{theo:max} is trivially generalizable to $kn+c$ cycles.

\par\medskip\noindent
\textbf{\textit{Generalization of theorem~\ref{theo:cycle_restriction_1}}}
\par\noindent
We generalize theorem~\ref{theo:cycle_restriction_1} for $kn+c$ cycles having $N_1$ odd and $N_0$ even members:

\begin{enumerate}[label=(\alph*)]
\item A cycle only exists if the inequality $2^{N_1+N_0}-k^{N_1}>0$ holds.
\item A cycle only exists if the integer $c$ and the difference $2^{N_1+N_0}-k^{N_1}$ are not coprime: $\gcd(c,2^{N_1+N_0}-k^{N_1})>1$.
\item Let $0\le x_1<x_2<\ldots<x_{N_1}<\le N_1-1$ be all positions (the indexing is zero-based) in the parity vector occupied by $1$. A cycle only exists if the divisibility $2^{N_1+N_0}-k^{N_1}\mid c\cdot z(s)$ holds, where $z$ is the function~\ref{eq:z}.
\item The number of $kn+c$ cycles is alsway less than or equal to the number of $kn+a\cdot c$ cycles, where $a$ is an odd number.
\end{enumerate}

\par\medskip\noindent
\textbf{\textit{Generalization of theorem~\ref{theo:cycle_uniqueness}}}
\par\noindent
We generalize theorem~\ref{theo:cycle_uniqueness} by stating that two different primitive cycles, $kn+c_1$ and $kn+c_2$, can never share a common parity vector.

\begin{proof}
A $kn+c$ cycle with a given parity vector first appears at:
\[
c=\frac{(k-2)(2^{N_1+N_0}-k^{N_0})}{\gcd(A,(k-2)(2^{N_1+N_0}-k^{N_0}))}
\]

Let there exist cycles $kn+c_1$ and $kn+c_1$ with the same parity vector, this implies that the values of $A$ and $(k-2)(2^{N_1+N_0}-k^{N_1})$ as defined in Definition \ref{def:odd_even_elements} are same for both the cycles. Therefore using the formula, a cycle can exist iff $v_1$ is an integer, i.e $c \cdot A$ divides $(k-2)(2^{N_1+N_0}-k^{N_1})$. The cycle will originate for the minimum such value of $c$. Therefore there can only be one value of $c$ for which the parity vector produces a cycle that is not inherited.
\end{proof}

\newpage
\noindent
\textbf{\textit{Generalizing the binary rotations to $\boldsymbol{kn+c}$ cycles}}
\par\noindent
Let $B$ be a binary number. The divisibility feature $l\mid N_1(B)\cdot r+1$, demonstrated in section~\ref{sec:binary_rotations_to_cycles} holds for the generalized $kn+c$ cycles. For this we set $c=(k-2)(2^l-k^{N_1(B)})$ and generalize function~\ref{eq:z} as follows:
\begin{equation}
\label{eq:func_z_k}
z_k(B)=\sum_{i=1}^{N_1}k^{N_1-i}2^{x_i}
\end{equation}
Here again $0\le x_1<x_2<\ldots<x_{N_1}\le N_1-1$ are the positions (indexing is zero-based) in $B$ occupied by $1$.

\par\medskip\noindent
\textbf{\textit{More theorems for $\boldsymbol{kn+c}$ cycles}}

A positive integer $k$ is called a \textit{Crandall number}, if there exists a $kn+1$ cycle and the following very fundamental theorem~\ref{theo:crandall_wieferich} is well known, see \cite{Ref_Crandall_1978}, \cite{Ref_Franco_Pomerance_1995}:

\begin{theorem}
\label{theo:crandall_wieferich}
Every Wieferich number is a Crandall number. In other words, if $k$ is a Wieferich number, then a cycle $kn+1$ cycle exists.
\end{theorem}

\par\noindent
Franco and Pomerance provided a proof for this theorem~\ref{theo:crandall_wieferich} in their paper \cite{Ref_Franco_Pomerance_1995}.

\begin{theorem}
\label{theo:co_primality}
If $c_1$ and $c_2$ are coprime, then for a given $k$ both functions $f_{k,c_1}$ and $f_{k,c_2}$ do not have any common non-trivial cycle (cycle with the same parity vector).
\end{theorem}

A proof is given by Anant Gupta \cite{Ref_Gupta_2020}. The idea can be sketched as follows: Let $i$ be an integer. Since $k^i$ does not divide $2^{N_1+N_0}-k^{N_1}$, all $kn+c$ cycles where $c=k^i$ will require $2^{N_1+N_0}-k^{N_1}$ to divide $A$ (recall that $A$ is specified by definition~\ref{def:odd_even_elements}), which is the same condition for $kx+1$ cycles. This implies that all cycles of $f_{k,k^i}$ are equal to the cycles of $f_{k,1}$. Similarly all cycles of $f_{k,c}$ are equal to the cycles of $f_{k,k^i\cdot c}$.

\section{Conclusion}
In this paper we investigated the behaviour of rotating binary numbers. We found that a rotation by $r$ digits to the left of a binary number $B$ exhibits in particular cases the divisibility \mbox{$l\mid N_1(B)\cdot r+1$}, where $l$ is the bit-length of $B$ and $N_1$ is the Hamming weight of $B$ and $r$ is the left-rotational distance. We investigated the connection between this rotational distance, the bit length and the Hamming weight. A core property is, that only under certain circumstances the above mentioned divisibility becomes true -- namely this divisibility occurs for cycles.

\newpage
\bibliographystyle{plain}
\bibliography{main}

\begin{thebibliography}{10}

\bibitem{Ref_Allouche_2003}
J.-P. Allouche and J.~Shallit.
\newblock {\em Automatic Sequences}.
\newblock Cambridge University Press, Cambridge, United Kingdom, 2003.

\bibitem{Ref_Cox_2012}
D.~Cox.
\newblock The 3n+1 problem: A probabilistic approach.
\newblock {\em Journal of Integer Sequences}, 15(5):1--11, 2012.

\bibitem{Ref_Cox_2021}
D.~Cox, S.~Ghosh, and E.~Sultanow.
\newblock Generalizing halbeisen's and hungerbühler's optimal bounds for the
  length of collatz cycles to 3n+c cycles.
\newblock {\em Journal of Mathematics and Computer Science}, 24(4):330--337,
  2021. arXiv:2101.04067 [math.GM].

\bibitem{Ref_Crandall_1978}
R.~E. Crandall.
\newblock On the "3x+1" problem.
\newblock {\em Mathematics of Computation}, 32(144):1281--1292, 1978.

\bibitem{Ref_Franco_Pomerance_1995}
Z.~Franco and C.~Pomerance.
\newblock On a conjecture of crandall concerning the "qx+1" problem.
\newblock {\em Mathematics of Computation}, 64(211):1333--1336, 1995.

\bibitem{Ref_Gupta_2020}
A.~Gupta.
\newblock On cycles of generalized collatz sequences, 2020. arXiv:2008.11103
  [math.NT].

\bibitem{Ref_Halbeisen_Hungerbuehler_1997}
L.~Halbeisen and N.~Hungerbuehler.
\newblock Optimal bounds for the length of rational collatz cycles.
\newblock {\em Acta Arithmetica}, 78(3):227--239, 1997.

\bibitem{Ref_Sedgewick_2011}
R.~Sedgewick and K.~Wayne.
\newblock {\em Algorithms}.
\newblock Addison-Wesley, Upper Saddle River, NJ, 4 edition, 2011.

\bibitem{Ref_Sultanow_2021}
E.~Sultanow.
\newblock Data science for number and coding theory: Divisibility, periodic
  sequences and discrete logarithm.
\newblock
  \url{https://www.slideshare.net/Sultanow/data-science-for-number-and-coding-theory-246114021},
  2021.
\newblock Slides from the talk "Data Science for Number and Coding Theory" at
  "Code Days 2021".

\bibitem{Ref_Weisstein_BitLength}
E.~W. Weisstein.
\newblock "bit length." from mathworld--a wolfram web resource.
\newblock \url{https://mathworld.wolfram.com/BitLength.html}.

\bibitem{Ref_Weisstein_DigitCount}
E.~W. Weisstein.
\newblock "digit count." from mathworld--a wolfram web resource.
\newblock \url{https://mathworld.wolfram.com/DigitCount.html}.

\bibitem{Ref_Weisstein_GDE}
E.~W. Weisstein.
\newblock "greatest dividing exponent." from mathworld--a wolfram web resource.
\newblock \url{https://mathworld.wolfram.com/GreatestDividingExponent.html}.

\end{thebibliography}

\newpage
{\renewcommand{\arraystretch}{1.8}
\begin{table}[H]
	\centering
	\begin{tabular}{|P{1.6cm} p{13.2cm}|}
		\hline
		\multicolumn{2}{|l|}{\thead[l]{\textbf{Fundamentals short and sweet}}}
		\\
		$B$ & We denote $B$ as a number in base-$2$ representation (a binary word) of bit-length $l=N_1(B)+N_0(B)$ consisting of $N_1(B)$ ones and $N_0(B)$ zeros.
		\\
		bit-length & The bit-length $l$ of an integer $n$ specifies the number of bits used for the binary representation of this integer. It is given by $l=\lfloor\log_2(n)\rfloor+1=\lceil\log_2(n+1)\rceil$, see \cite{Ref_Weisstein_BitLength}.
		\\
		$\gde(n,2)$ & The greatest dividing exponent of base $2$ with respect to a number $n$ is the largest integer value of $k$ such that $2^k\mid n$, where $2^k\le n$, see \cite{Ref_Weisstein_GDE}.
		\\
		digit count $N_d^{2}(B)$ & The number $N_d^{2}(B)=N_d(B)$ of digits $d$ in the base-$2$ representation of the number $B$ is called the binary digit count for $d$. Thus $N_1(B)$ specifies the number of ones in $B$ (also termed as \textit{Hamming weight} of $B$) given by the difference $B-\gde(B!,2)$. Analogously, $N_0(B)$ specified the number of zeros in $B$ \cite{Ref_Weisstein_DigitCount}.
		\\
		rotate a binary & The left rotation (left circular shift) of a binary $B$ by $r$ bits is the function $\lambda_{left}(B,r,l)=(B\cdot2^r)\bmod{(2^l-1)})$, where $l$ is the bit-length of $B$. The right rotation is given by $\lambda_{right}(B,r,l)=\lambda_{left}(B,l-r,l)$. The bit-length is implicitly given and we can use shorter $\lambda_{left}(B,r)$.
		\\
		rotational distance & The left-rotational distance of a binary $B_2$ from the binary $B_1$ is the required amount of rotating $B_1$ (bit by bit) until the rotated binary matches $B_2$. The right-rotational distance is defined analogously.
		\\
		$3n+c$ cycle & We consider the function $f_c(x)$ given by equation~\ref{eq:func_collatz} and call a cycle the sequence of distinct positive integers $(v_1,v_2,\ldots,v_l)$ where $f_c(v_1)=v_2$ and $f_c(v_2)=v_3$ and so forth and finally $f_c(v_{l+1})=v_1$.
		\\
		$kn+c$ cycle & We generalize $3n+c$ cycles by replacing $3$ with any positive integer $k$.
		\\
		periodic sequence & Let $(v_1,v_2,\ldots,v_l)$ be a $3n+c$ cycle. We call a sequence that forms a repetition of this cycle $(v_1,v_2,\ldots,v_l,v_1,v_2,\ldots,v_l,\ldots)$ periodic.
		\\
		Primitive cycle & If all members of a $3n+c$ cycle share a same common divisor greater than one, then this cycle is referred to as a \textit{non-primitve} cycle, otherwise it is a \textit{primitve} cycle, see \cite{Ref_Cox_2021}.
        \\ \hline
	\end{tabular}
\end{table}}

{\renewcommand{\arraystretch}{1.8}
\begin{table}[H]
	\centering
	\begin{tabular}{|P{1.6cm} p{13.2cm}|}
		\hline
		Parity vector & The parity vector of a $3n+c$ cycle $(v_1,v_2,\ldots,v_l)$ is a binary vector having $l=N_1+N_0$ entries -- a $1$ at position $i$, if $v_i$ is odd, and otherwise $0$.
		\\
        Non-reduced word & Let us consider a $3n+c$ cycle with $N_1$ odd and $N_0$ even members. The non-reduced word describing this cycle is a word of length $N_1+N_0$ over the alphabet $\{n_1,n_0\}$, which has a $n_1$ at those positions, where an odd member and a $n_0$ where an even member is located in the cycle. For instance, we treat the word $n_1n_0n_1n_1n_0n_1n_0n_1$ synonymous to the parity vector $(1,0,1,1,0,1,0,1)$ or even simpler to the binary sequence (binary word) $10110101$.
		\\ \hline
	\end{tabular}
\end{table}}

\end{document}